\documentclass[11pt,letterpaper,reqno]{amsart}

\usepackage[left=1.2in,right=1.2in]{geometry}
\usepackage{amsmath}
\usepackage{amsthm}
\usepackage{amsfonts}
\usepackage{amssymb}
\usepackage{amsrefs}
\usepackage[colorlinks]{hyperref}
\usepackage{}
\allowdisplaybreaks

\hypersetup{
linkcolor=blue,         
citecolor=blue,        
urlcolor=blue
}

\newtheorem{theorem}{Theorem}[section]
\newtheorem{proposition}{Proposition}[section]
\newtheorem{lemma}{Lemma}[section]

\theoremstyle{definition}
\newtheorem{definition}{Definition}[section]
\newtheorem{remark}{Remark}[section]

\newcommand{\PSL}{\operatorname{PSL}}

\newcommand{\dist}{\operatorname{dist}}
\newcommand{\Stab}{\operatorname{Stab}}

\begin{document}
\title{On the density of some sparse horocycles}

\author{Cheng Zheng}
\address{School of Mathematical Sciences, Shanghai Jiao Tong University, China}
\email{zheng.c@sjtu.edu.cn}

\subjclass[2020]{Primary 37A17; Secondary 11J99}
\keywords{density of sparse horocycles, Diophantine approximation, sparse equidistribution, Ratner's theorem, almost primes and sieve}

\thanks{The author acknowledges the support of ISF grants number 662/15 and 871/17, the support at Technion by a Fine Fellowship, and the support by Institute of Modern Analysis-A Frontier Research Center of Shanghai. This work has received funding from the European Research Council (ERC) under the European Union's Horizon 2020 research and innovation programme (Grant Agreement No. 754475).}

\begin{abstract}
Let $\Gamma$ be a non-uniform lattice in $\PSL(2,\mathbb R)$. In this note, we show that there exists a constant $\gamma_0>0$ such that for any $0<\gamma<\gamma_0$, any one-parametrer unipotent subgroup $\{u(t)\}_{t\in\mathbb R}$ and any $p\in\PSL(2,\mathbb R)/\Gamma$ which is not $u(t)$-periodic, the orbit $\{u(n^{1+\gamma})p:n\in\mathbb N\}$ is dense in $\PSL(2,\mathbb R)/\Gamma$. We also prove that there exists $N\in\mathbb N$ such that for the set $\Omega(N)$ of $N$-almost primes, and for any $p\in\PSL(2,\mathbb R)/\Gamma$ which is not $u(t)$-periodic, the orbit $\{u(x)p:x\in\Omega(N)\}$ is dense in $\PSL(2,\mathbb R)/\Gamma$.
\end{abstract}

\maketitle

\section{Introduction}\label{intro}
The theory of equidistribution of unipotent orbits has been studied extensively over the past few decades, and many fundamental tools have been established, one of which is the celebrated Ratner's Theorem \cite{D,M1,R1,R,R2,MT}. Using Ratner's Theorem, we can investigate the distribution of a subset in a homogeneous space which is (almost) invariant under certain unipotent group action. On the other hand, for subsets which do not exhibit any invariance of group actions, their distributions in homogeneous spaces are still far from being well-known. An interesting case is the following sparse equidistribution problem: Let $\{u(t)\}_{t\in\mathbb R}$ be a unipotent flow on a homogeneous space $G/\Gamma$ and consider the parameter $t$ sampled in a discrete subset $\Omega$ of density zero in $\mathbb R$. Then one seeks to establish results about the equidistribution of $\{u(t)p:t\in\Omega\}$ in $G/\Gamma$. This sparse equidistribution problem has attracted much attention recently \cite{Sh, V, B1, GT, K1, K2, L}, and it is listed as one of the open problems in homogeneous dynamics \cite{G}.

In this note, we investigate the topological version of the sparse equidistribution problem in the homogeneous space $G/\Gamma=\PSL(2,\mathbb R)/\Gamma$. Specifically, let $\{u(t)\}_{t\in\mathbb R}$ be a one-parameter unipotent subgroup of $\PSL(2,\mathbb R)$, and suppose that the discrete subset $\Omega$ is either $\{n^{1+\gamma}:n\in\mathbb N\}$ $(\gamma>0)$ or a set of almost primes. We will consider the set $\{u(t)p:t\in\Omega\}$ $(\forall p\in G/\Gamma)$ and study its density in $\PSL(2,\mathbb R)/\Gamma$.

So far a few results about this problem have been established. For instance, in the case when $\Gamma$ is a uniform lattice in $\PSL(2,\mathbb R)$, if $\Omega$ is the subset $\{n^{1+\gamma}:n\in\mathbb N\}$ for a sufficiently small constant $\gamma>0$, then it is shown \cite{V, TV} that $\{u(t)p:t\in\Omega\}$ is uniformly distributed in $\PSL(2,\mathbb R)/\Gamma$. If $\Omega$ is a subset of almost primes, then $\{u(t)p:t\in\Omega\}$ is uniformly distributed as well \cite{M2}. Consequently, the set $\{u(t)p: t\in\Omega\}$ is dense in $\PSL(2,\mathbb R)/\Gamma$. 

As for non-uniform lattices in $\PSL(2,\mathbb R)$, the only well-studied case is $\Gamma=\PSL(2,\mathbb Z)$. Indeed, when $\Gamma=\PSL(2,\mathbb Z)$ and $\Omega$ is the set of primes, it is proved \cite{SU} that for any $p$ which is not $u(t)$-periodic, any limiting distribution of $\{u(t)p:t\in\Omega\}$ is absolutely continuous with respect to the invariant probability measure on $\PSL(2,\mathbb R)/\PSL(2,\mathbb Z)$. If $\Omega$ is a set of almost primes, the problem about the distribution of the orbit $\{u(t)p:t\in\Omega\}$ is proposed in \cite{SU}, and  it is proved in \cite{M2} that $\{u(t)p:t\in\Omega\}$ is dense in $\PSL(2,\mathbb R)/\PSL(2,\mathbb Z)$ provided that $p$ satisfies certain Diophantine condition. On the other hand, for a general non-uniform lattice $\Gamma$, to the best of author's knowledge, the only relevant reference is \cite{Z} where it is shown that if $\Omega$ is the set $\{n^{1+\gamma}:n\in\mathbb N\}$ for a sufficiently small constant $\gamma>0$ and $p$ is a Diophantine point in $\PSL(2,\mathbb R)/\Gamma$, then $\{u(t)p:t\in\Omega\}$ is uniformly distributed in $\PSL(2,\mathbb R)/\Gamma$. Consequently, $\{u(t)p:t\in\Omega\}$ is dense in $\PSL(2,\mathbb R)/\Gamma$ under the restriction that $p$ is Diophantine. We remark that the set of Diophantine points is a proper subset in the set of points which are not $u(t)$-periodic.

In this paper, we establish some results about the topological version of the sparse equidistribution problem in $\PSL(2,\mathbb R)/\Gamma$ for any non-uniform lattice $\Gamma$ and any initial point $p$ which is not $u(t)$-periodic. These extend some of the results in \cite{M2, Z} in the topological sense. Our method relies on a Diophantine analysis on the initial point $p$. If $p$ is Diophantine of type $\kappa$ for some $\kappa>0$, then the orbit $\{u(t)p\}_{t\in\mathbb R}$ exhibits certain chaos and we can use effective Dani-Smillie theorem (a special case of Ratner's theorem) to show that the orbit $\{u(t)p:t\in\Omega\}$ is dense. If $p$ is not Diophantine of type $\kappa$, then the orbit $\{u(t)p\}_{t\in\mathbb R}$ can be approximated by a sequence of periodic $u(t)$-orbits, and we can apply equidistribution results in circles in this case. We remark that some part of our analysis may appear implicitly in \cite{SU} in a different languange. 

Now we state our main results. Let $G=\PSL(2,\mathbb R)$ and $\Gamma$ a non-uniform lattice in $G$. An element $g\in G$ is unipotent if $1$ is the only eigenvalue of $g$. A one-parameter unipotent subgroup $\{u(t)\}_{t\in\mathbb R}$ of $G$ is a one-parameter subgroup of $G$ where $u(t)$ is unipotent for any $t\in\mathbb R$.  In this paper, we prove the following

\begin{theorem}\label{th11}
Let $\Gamma$ be a non-uniform lattice in $\PSL(2,\mathbb R)$. Then there exists a constant $\gamma_0>0$ such that for any $0<\gamma<\gamma_0$, any one-parameter unipotent subgroup $\{u(t)\}_{t\in\mathbb R}$ and any $p\in\PSL(2,\mathbb R)/\Gamma$ which is not $\{u(t)\}_{t\in\mathbb R}$-periodic, the orbit $\{u(n^{1+\gamma})p:n\in\mathbb N\}$ is dense in $\PSL(2,\mathbb R)/\Gamma$.
\end{theorem}

Recall that a natural number $x$ is said to be an $N$-almost prime for some $N\in\mathbb N$ if there are at most $N$ prime factors (counted with multiplicities) in the prime factorization of $x$. We denote by $\Omega(N)$ the set of $N$-almost primes in $\mathbb N$. Using the same techniques as in the argument of Theorem~\ref{th11}, we can prove the following

\begin{theorem}\label{th12}
Let $\Gamma$ be a non-uniform lattice in $\PSL(2,\mathbb R)$. Then there exists $N\in\mathbb N$ such that for any one-parameter unipotent subgroup $\{u(t)\}_{t\in\mathbb R}$ and any $p\in\PSL(2,\mathbb R)/\Gamma$ which is not $\{u(t)\}_{t\in\mathbb R}$-periodic, the orbit $\{u(x)p:x\in\Omega(N)\}$ is dense in $\PSL(2,\mathbb R)/\Gamma$.
\end{theorem}



The strategy of the proof of Theorem~\ref{th11} is the following: If $p$ is Diophantine of type $\kappa$ for some $\kappa>0$, then the result is a corollary of \cite{Z}. So it suffices to deal with the case when $p$ is not Diophantine of type $\kappa$. In this case, we can find a sequence of $\{u(t)\}_{t\in\mathbb R}$-periodic points $\{q_k\}_{k\in\mathbb N}$ converging to $p$, and the orbit $\{u(n^{1+\gamma})p:n\in\mathbb N\}$ can be approximated by the orbit $\{u(n^{1+\gamma})q_k:n\in\mathbb N\}$ as $k\to\infty$. Note that the orbit $\{u(n^{1+\gamma})q_k:n\in\mathbb N\}$ is a subset in the circle $\{u(t)q_k:t\in\mathbb R\}$. It is known that as $k$ goes to infinity, the circle $\{u(t)q_k\}_{t\in\mathbb R}$ becomes dense in the homogeneous space $\PSL(2,\mathbb R)/\Gamma$, while, using Fourier analysis, we can also prove that the orbit $\{u(n^{1+\gamma})q_k:n\in\mathbb N\}$ becomes dense in the circle $\{u(t)q_k\}_{t\in\mathbb R}$. Then by standard approximation argument, we can show that the orbit $\{u(n^{1+\gamma})p:n\in\mathbb N\}$ is dense in $\PSL(2,\mathbb R)/\Gamma$. The strategy of the proof of Theorem \ref{th12} is similar.\\


\section{Notation and Preliminaries}\label{pre}
In this section, we introduce the notation and preliminaries needed in the paper. Let $$U=\left\{u_0(s)=\left(\begin{array}{cc}1 & s \\0 & 1\end{array}\right): s\in\mathbb R\right\}$$ be the upper triangular unipotent subgroup in $\PSL(2,\mathbb R)$. Then any one-parameter unipotent subgroup of $\PSL(2,\mathbb R)$ is conjugate to $\{u_0(s):s\in\mathbb R\}$ or $\{u_0(-s):s\in\mathbb R\}$. We denote by $U_-$ the lower triangular unipotent subgroup in $G$, $A$ the diagonal group $$A=\left\{a_t=\left(\begin{array}{cc} e^{-t/2} & 0\\ 0 & e^{t/2}\end{array}\right):t\in\mathbb R\right\}$$ and $$\omega=\left(\begin{array}{cc} 0 & -1\\ 1 & 0\end{array}\right)$$ the representative of the non-trivial Weyl element in $\PSL(2,\mathbb R)$. In the following, for any two quantities $c_1, c_2>0$, we write $c_1\ll c_2$ if there exists a constant $C>0$ such that $c_1\leq Cc_2$. If $c_1\ll c_2$ and $c_2\ll c_1$, then we write $c_1\sim c_2$. We will specify the implicit constants in the context if necessary. 

We denote by $C^\infty(G/\Gamma)$ the space of all smooth functions on $G/\Gamma$, and $C_c^\infty(G/\Gamma)$ the space of all compactly supported smooth functions on $G/\Gamma$. For $f\in C^\infty(G/\Gamma)$, let $\|f\|_{p,k}$ be the Sobolev $L^p$-norm of $f$ involving all the Lie derivatives of order up to $k$ (with respect to a basis of the Lie algebra of $G$). For any smooth function $f$ on $\mathbb R$, we will use the same notation $\|f\|_{p,k}$ for the $L^p$-norm of $f$ involving derivatives of order $\leq k$. Note that $\|f\|_{\infty,0}$ is the supremum norm of $f$. 

Since $\Gamma$ is a non-uniform lattice in $G$, it is known that there exist $U_-$-periodic points in $G/\Gamma$. Without loss of generality, we may assume that $e\Gamma$ is a $U_-$-periodic point. Now we recall several notions about Diophantine points in homogeneous spaces \cite{CZ,Z}. For any $p\in G/\Gamma$, we denote by $\Stab(p)$ the stabilizer of $p$ in $G$. If $p=g\Gamma$, then $\Stab(p)=g\Gamma g^{-1}$. We fix a norm $\|\cdot\|_{\mathfrak g}$ on the Lie algebra $\mathfrak g$ of $G$, and denote by $d_G(\cdot,\cdot)$, $d_{U_-}(\cdot,\cdot)$ and $d_{G/\Gamma}(\cdot,\cdot)$ the induced distances on $G$,  $U_-(e\Gamma)$ and $G/\Gamma$ by $\|\cdot\|_{\mathfrak g}$ respectively.

\begin{definition}[\cite{CZ,Z}]\label{d21}
For any $p\in G/\Gamma$, we define the injectivity radius at $p$ by $$\eta(p)=\inf_{v\in\textup{Stab}(p)\setminus\{e\}}d_G(v,e).$$
\end{definition}

\begin{definition}[\cite{CZ,Z}]\label{d22}
A point $p\in G/\Gamma$ is Diophantine of type $\kappa$ (with respect to the diagonal group $\{a_t\}_{t\in\mathbb R}$) if there exists a constant $C>0$ such that $$\eta(a_tp)\geq Ce^{-\kappa t}\textup{ for all }t>0.$$ Note that $0\leq\kappa\leq1$.
\end{definition}

\begin{definition}[\cite{CZ,Z}]\label{d23a}
A point $p\in G/\Gamma$ is called rational if $\Stab(p)\cap U\neq\{e\}$. Note that in this case, $p$ is a $U$-periodic point in $\PSL(2,\mathbb R)/\Gamma$ and $U\cdot p\cong U/(\Stab(p)\cap U)$.
\end{definition}

\begin{definition}[\cite{CZ,Z}]\label{d24a}
Define the denominator of a rational point $p\in G/\Gamma$ by $$d(p)=\inf_{v\in(\Stab(p)\cap U)\setminus\{e\}}\|\log v\|_{\mathfrak g}.$$ Here $\log$ is the inverse of the exponential map $\exp$ from $U$ to the Lie algebra of $U$. Note that for $G=\PSL(2,\mathbb R)$, we can (and will) choose the norm $\|\cdot\|_{\mathfrak g}$ so that $d(p)$ is the period of the $U$-periodic orbit $U\cdot p$ in $\PSL(2,\mathbb R)/\Gamma$.
\end{definition}

The following proposition plays an important role in our analysis.

\begin{proposition}{\cite[Proposition 6.3]{CZ}}\label{p22}
Let $p\in U_-(e\Gamma)\subset G/\Gamma$. If $p$ is not Diophantine of type $\kappa$ ($0<\kappa<1$), then there exist a constant $C>0$ depending only on $G/\Gamma$ and $\kappa$, and a sequence of distinct $U$-periodic points (or equivalently, rational points) $q_k\in U_-(e\Gamma)$ with $U$-period (or equivalently, the denominator) $d(q_k)\to\infty$ such that $$d_{U_-}(p,q_k)\leq Cd(q_k)^{-\frac1{1-\kappa}}.$$ 
\end{proposition}

We also need the structure of the fundamental domain of $G/\Gamma$. 
\begin{theorem}[\cite{GR}]\label{domains}
Let $K=\operatorname{PSO}(2)$. Then there exist $\alpha_0\in\mathbb R$, a compact subset $\Omega_0\subset\mathbb R$ and a finite set $\Delta=\{\sigma_j\}_{j=1}^n\subset G$ such that for the sets $$A_{\alpha_0}=\left\{\left(\begin{array}{cc} e^{s/2} & 0 \\ 0 & e^{-s/2}\end{array}\right)|s\leq\alpha_0\right\}\textup{ and }U_{\Omega_0}=\left\{\left(\begin{array}{cc} 1 & s \\ 0 & 1\end{array}\right)| s\in\Omega_0\right\}$$ we have
\begin{enumerate}
\item $G=\bigcup_{j=1}^nKA_{\alpha_0}U_{\Omega_0}\sigma_j\Gamma$.
\item $\sigma_j\Gamma\sigma_j^{-1}\cap U$ is a cocompact lattice in $U$ for any $1\leq j\leq n$.
\item $U_{\Omega_0}$ contains a fundamental domain of $U/(U\cap\sigma_j\Gamma\sigma_j^{-1})$.
\end{enumerate}
\end{theorem}

In this paper, sometimes we write the homogeneous space $G/\Gamma$ as the space of right $\Gamma$-cosets $\Gamma\backslash G$. The correspondence between $G/\Gamma$ and $\Gamma\backslash G$ is the following: If $p=g\Gamma$, then $g\Gamma$ maps to $\Gamma g^{-1}=(g\Gamma)^{-1}$ and we denote by $p^{-1}=\Gamma g^{-1}$.

\section{Auxilliary lemmas}
Now we discuss a bit about the group action of $G$ on the upper half plane $\mathcal H^2$. It is known that $G$ acts on the upper half plane $\mathcal H^2$ by
\begin{align*}
\left(\begin{array}{cc} a & b\\ c & d\end{array}\right)\cdot z:=\frac{az+b}{cz+d}
\end{align*}
and we denote by $\pi$ the projection map from $\Gamma\backslash G$ (the space of right $\Gamma$-cosets) to $\Gamma\backslash\mathcal H^2$ by $\pi(\Gamma g)=\Gamma g\cdot i$. We define the geodesic flow on $\Gamma\backslash G$ by
\begin{align*}
g_t(\Gamma g)=\Gamma g\left(\begin{array}{cc} e^{t/2} & 0\\ 0 & e^{-t/2}\end{array}\right)=\Gamma g\cdot a_{-t}
\end{align*}
Fix a point $p_0\in\Gamma\backslash\mathcal H^2$. For any $x\in\Gamma\backslash G$, define $$\dist(x)=d_{\Gamma\backslash\mathcal H^2}(p_0,\pi(x))$$ where $d_{\Gamma\backslash\mathcal H^2}(\cdot,\cdot)$ is the hyperbolic distance on $\Gamma\backslash\mathcal H^2$. The following lemma describes the relation between $\dist$ and $\eta$.

\begin{lemma}\label{auxilliary}
For any $p\in G/\Gamma$, we have $$e^{\dist(p^{-1})}\sim\frac1{\eta(p)}.$$ The implicit constant depends only on $G/\Gamma$.
\end{lemma}
\begin{proof}
By Theorem~\ref{domains}, we can write $$G/\Gamma=\bigcup_{j=1}^nKA_{\alpha_0}U_{\Omega_0}\sigma_j\Gamma/\Gamma.$$ Let $p=kan\sigma_j\Gamma$ for some $k\in K=\operatorname{PSO}(2)$, $a=\left(\begin{array}{cc} e^{s/2} & 0\\ 0 & e^{-s/2}\end{array}\right)\in A_{\alpha_0}$ and $n\in U_{\Omega_0}$. Then we can compute that $$\eta(p)\sim e^s$$ where the implicit constant depends only on $K$, $\Gamma$, $\sigma_j$ and $\alpha_0$. On the other hand, we know that $$p^{-1}=\Gamma\sigma_j^{-1}n^{-1}a^{-1}k^{-1},\quad\pi(p^{-1})=\Gamma\sigma_j^{-1}n^{-1}a^{-1}\cdot i$$ and we can compute that $$e^{\dist(p^{-1})}\sim e^{-s}$$ where the implicit constant depends only on $\alpha_0$, $\Omega_0$, $\Gamma$ and $\sigma_j$. Then it follows that $$e^{\dist(p^{-1})}\sim\frac1{\eta(p)}.$$ This completes the proof of the lemma.
\end{proof}

\begin{lemma}\label{l31}
Let $p\in G/\Gamma$. If $p$ is Diophantine of type $\kappa$ $(0<\kappa<1)$, then there exists a constant $C>0$ depending only on $G/\Gamma$ and $p$ such that for any $t>0$ we have $$e^{\dist(g_t(p^{-1}))}\leq Ce^{\kappa t}.$$
\end{lemma}
\begin{proof}
Note that by definition we have $$g_t(p^{-1})=p^{-1}\cdot a_{-t}=(a_t\cdot p)^{-1}.$$ The lemma then follows from Lemma~\ref{auxilliary}.
\end{proof}
\begin{remark}\label{r31}
By combining Lemma~\ref{l31} and \cite[Lemma~3.4]{Z}, we can also conclude that $p^{-1}$ is Diophantine of type $((1+\kappa)/(1-\kappa),(1+\kappa)/(1-\kappa),\dots,(1+\kappa)/(1-\kappa))$ in the sense of \cite[Definition 1.1]{Z}. We will use this fact in the proof of Theorem~\ref{th11}.
\end{remark}

\section{Proof of Theorem~\ref{th11}}\label{sparse}
In this section, we prove Theorem~\ref{th11}. We know that any one-parameter unipotent subgroup is conjugate to $\{u_0(s):s\in\mathbb R\}$ or $\{u_0(-s):s\in\mathbb R\}$. Let $$a_0=\left(\begin{array}{cc} 1 & 0\\ 0 & -1\end{array}\right)$$ and define $\phi:G\to G$ by $$\phi(g)=a_0ga_0^{-1}.$$  Then $\phi$ is an automorphism of $G$, maps $u_0(s)$ to $u_0(-s)$ and defines a homeomorphism $\bar\phi$ between $G/\Gamma$ and $G/\phi(\Gamma)$. Moreover, for any $p\in G/\Gamma$, the subset $\{u_0(n^{1+\gamma})p:n\in\mathbb N\}$ is dense in $G/\Gamma$ if and only if $$\{u_0(-n^{1+\gamma})\phi(p):n\in\mathbb N\}=\bar\phi(\{u_0(n^{1+\gamma})p:n\in\mathbb N\})$$ is dense in $G/\phi(\Gamma)$. Therefore, without loss of generality, we may assume that the one-parameter unipotent subgroup is equal to $\{u_0(-s):s\in\mathbb R\}$. 

Now suppose that the initial point $p\in G/\Gamma$ is Diophantine of type $\kappa=1-1/{100}$. Then by Remark~\ref{r31} and \cite[Theorem 1.1]{Z}, there exists a constant $\gamma_0>0$ such that for any $0<\gamma<\gamma_0$ the following set $$\{p^{-1}\cdot u_0(n^{1+\gamma}):n\in\mathbb N\}$$ is dense in $\Gamma\backslash G$. It follows that $\{u_0(-n^{1+\gamma})\cdot p:n\in\mathbb N\}$ is dense in $G/\Gamma$. Therefore, we only need to consider the case when $p$ is not Diophantine of type $1-1/{100}$.

By Bruhat decomposition, we know that $$G=UAU_-\cup\omega AU_-\textup{ and }G/\Gamma=UAU_-/\Gamma\cup\omega AU_-/\Gamma.$$ As $e\Gamma$ is a $U_-$-periodic point and any point in $\omega AU_-/\Gamma$ is a $U$-periodic point, to prove Theorem~\ref{th11}, we may assume that $p\in UAU_-/\Gamma$. Now write $p=uan\Gamma$ for $u\in U$, $a\in A$ and $n\in U_-$. Since $a$ normalizes $U$, we have $$u_0(t)\cdot p=u_0(t)\cdot uan\Gamma=(ua)\cdot u_0(ct)\cdot n\Gamma\quad(\forall t\in\mathbb R)$$ for some constant $c>0$. Besides, $p$ is Diophantine of type $\kappa=1-1/100$ if and only if $n\Gamma$ is Diophantine of type $\kappa=1-1/100$. Hence to prove Theorem~\ref{th11}, it suffices to prove the following

\begin{theorem}\label{th21}
Let $\Gamma$ be a non-uniform lattice in $\PSL(2,\mathbb R)$. Then there exists a constant $\gamma_0>0$ such that for any $p\in U_-(e\Gamma)$ which is neither $U$-periodic nor Diophantine of type $\kappa=1-1/100$, any $0<\gamma<\gamma_0$ and any $c>0$, the orbit $\{u_0(-c n^{1+\gamma})p:n\in\mathbb N\}$ is dense in $\PSL(2,\mathbb R)/\Gamma$.
\end{theorem}

The rest of this section is devoted to the proof of Theorem~\ref{th21}. The following proposition is well-known.

\begin{proposition}{\cite[Corollary 8.12]{IK}}\label{p21}
Let $f$ be a smooth function on $[a,b]$ with $b-a\geq1$ such that $f''(x)\geq\delta>0$. Then $$\left|\sum_{n=a}^be^{2\pi if(n)}\right|\ll\frac{f'(b)-f'(a)+1}{\sqrt{\delta}}$$ where the implicit constant is absolute.
\end{proposition}

As a corollary of Proposition~\ref{p21}, we deduce the following
\begin{lemma}\label{l21}
Let $c>0$ and $0<\gamma<1$. Let $f(x)=cx^{1+\gamma}$. For each $k\in\mathbb R\setminus\{0\}$, we have $$\left|\sum_{n=1}^Ne^{2\pi ikf(n)}\right|\ll |k|^{\frac12}N^{(1+\gamma)/2}+|k|^{-\frac12}N^{(1-\gamma)/2}$$ where the implicit constant depends only on $c$ and $\gamma$.
\end{lemma}
\begin{proof}
Without loss of generality, we may assume that $k>0$. (Otherwise we may apply complex conjugation.) Let $a=1$ and $b=N$ in Proposition~\ref{p21}, and we take $$\delta=|ck|(1+\gamma)\gamma N^{\gamma-1}.$$ Then we have $$(kf(x))''\geq\delta>0$$ and by Proposition~\ref{p21}
\begin{align*}
\left|\sum_{n=1}^N e^{2\pi ikf(n)}\right|\ll&\frac{|ck|(1+\gamma)N^\gamma-|ck|(1+\gamma)+1}{\sqrt{|ck|(1+\gamma)\gamma N^{\gamma-1}}}\\
\ll&|k|^{\frac12}N^{(1+\gamma)/2}+|k|^{-\frac12}N^{(1-\gamma)/2}.
\end{align*}
This completes the proof of the lemma.
\end{proof}

\begin{lemma}\label{l22}
Let $c>0$ and $0<\gamma<1$. Let $f(x)$ be a smooth periodic function on $\mathbb R$ with period $l>1$. Then $$\left|\sum_{n=1}^Nf(-cn^{1+\gamma})-N\cdot\frac1l\int_0^lf(x)dx\right|\ll l^3 N^{(1+\gamma)/2}\|f\|_{\infty,2}.$$ Here the implicit constant depends only on $c$ and $\gamma$.
\end{lemma}
\begin{proof}
Write the Fourier series of $f$ on $[0,l]$ as $$f(x)=\sum_{k\in\mathbb Z} a_ke^{2\pi ikx/l}$$ where $a_k=\frac1l\int_0^l f(x)e^{-2\pi ikx/l}dx.$ Since $f\in C^\infty(\mathbb R)$, using integration by parts, we have $$|a_k|\leq\frac{l^2}{4\pi^2k^2}\|f\|_{\infty,2}.$$
Therefore, by Lemma~\ref{l21}, we have
\begin{align*}
&\left|\sum_{n=1}^N f(-cn^{1+\gamma})-N\cdot\frac1l\int_0^lf(x)dx\right|\\
\leq& \sum_{k\neq0} |a_k|\left|\sum_{n=1}^Ne^{-2\pi ik cn^{1+\gamma}/l}\right|\\
\ll&\sum_{k\neq 0}|a_k|N^{(1+\gamma)/2}|k|^{1/2}l^{1/2}\\
\ll& l^3N^{(1+\gamma)/2}\|f\|_{\infty,2}.
\end{align*}
This completes the proof of the lemma.
\end{proof}

\begin{lemma}\label{l23}
Let $c>0$ and $0<\gamma<1$. Let $f\in C_c^\infty(\PSL(2,\mathbb R)/\Gamma)$ and $p,q\in U_-(e\Gamma)$ with $$p=\left(\begin{array}{cc}1 & 0 \\x & 1\end{array}\right)\Gamma\textup{ and }q=\left(\begin{array}{cc}1 & 0 \\y & 1\end{array}\right)\Gamma.$$ Then
\begin{align*}
\left|\sum_{n=1}^Nf\left(u_0(-cn^{1+\gamma})p\right)-\sum_{n=1}^Nf\left(u_0(-cn^{1+\gamma})q\right)\right|\ll N^{3+2\gamma}|x-y|\|f\|_{\infty,1}.
\end{align*}
Here the implicit constant depends only on $c$ and $\gamma$.
\end{lemma}
\begin{proof}
Note that 
\begin{align*}
\left(\begin{array}{cc}1 & s \\0 & 1\end{array}\right)\left(\begin{array}{cc}1 & 0 \\\delta & 1\end{array}\right)\left(\begin{array}{cc}1 & -s \\0 & 1\end{array}\right)=\left(\begin{array}{cc}1+s\delta & -s^2\delta \\\delta & 1-\delta s\end{array}\right)
\end{align*}
and
\begin{align*}
u_0(s)q=&\left(\begin{array}{cc}1 & s \\0 & 1\end{array}\right)\left(\begin{array}{cc}1 & 0 \\y-x & 1\end{array}\right)\left(\begin{array}{cc}1 & -s \\0 & 1\end{array}\right)u_0(s)p\\
=&\left(\begin{array}{cc}1+s(y-x) & -s^2(y-x) \\ y-x & 1-(y-x)s\end{array}\right)u_0(s)p.
\end{align*}
Hence $$d_{G/\Gamma}\left(u_0(s)q,u_0(s)p\right)\leq \max\{s^2|x-y|,|s(x-y)|,|x-y|\}.$$ If $s=-cn^{1+\gamma}$ $(n\in\mathbb N)$, then $$d_{G/\Gamma}\left(u_0(s)q,u_0(s)p\right)\ll n^{2+2\gamma}|x-y|$$ where the implicit constant depends only on $c$.
The lemma then follows from mean value theorem.
\end{proof}

By combining Lemma~\ref{l22} and Lemma~\ref{l23}, we have the following

\begin{proposition}\label{p23}
Let $c>0$ and $0<\gamma<1$. Let $f\in C_c^\infty(\PSL(2,\mathbb R)/\Gamma)$ and $p\in U_-(e\Gamma)\subset\PSL(2,\mathbb R)/\Gamma$. Suppose that $p$ is neither a $U$-periodic point nor Diophantine of type $\kappa=1-1/100$. Let $\{q_k\}_{k\in\mathbb N}$ be a sequence of $U$-periodic points in $U_-(e\Gamma)$ which converges to $p$ as in Proposition~\ref{p22}. Then there is a constant $C>0$ such that
\begin{align*}
&\left|\frac1N\sum_{n=1}^Nf\left(u_0(-cn^{1+\gamma})p\right)-\frac1{d(q_k)}\int_0^{d(q_k)} f\left(u_0(s)q_k\right)ds\right|\\
\leq& C (N^{2+2\gamma}d_{U_-}(p,q_k)+d(q_k)^3N^{(\gamma-1)/2})\|f\|_{\infty,2}
\end{align*}
Here the constant $C$ depends only on $c$ and $\gamma$.
\end{proposition}
\begin{proof}
By Lemma~\ref{l23}, we have
\begin{align*}
\left|\frac1N\sum_{n=1}^Nf(u_0(-cn^{1+\gamma})p)-\frac1N\sum_{n=1}^Nf(u_0(-cn^{1+\gamma})q_k)\right|\ll N^{2+2\gamma}d_{U_-}(p,q_k)\|f\|_{\infty,1}.
\end{align*}
Since $f(u_0(x)q_k)$ is a smooth periodic function on $\mathbb R$ with period $d(q_k)$, by Lemma~\ref{l22}, we have
\begin{align*}
\left|\frac1N\sum_{n=1}^Nf(u_0(-cn^{1+\gamma})q_k)-\frac1{d(q_k)}\int_0^{d(q_k)}f(u_0(s)q_k)ds\right|\ll d(q_k)^3N^{\frac{\gamma-1}2}\|f\|_{\infty,2}.
\end{align*}
The proposition then follows from these two inequalities.
\end{proof}

\begin{proof}[Proof of Theorem~\ref{th21}]
Let $p\in\PSL(2,\mathbb R)/\Gamma$ be a point in $U_-(e\Gamma)$ which is neither $U$-periodic nor Diophantine of type $\kappa=1-1/100$. Then by Proposition~\ref{p22}, there exists a sequence of $U$-periodic points $\{q_k\}_{k\in\mathbb N}$ such that $q_k\to p$, $d(q_k)\to\infty$ and $$d_{U_-}(p,q_k)\leq d(q_k)^{-50}.$$ Now according to Proposition~\ref{p23}, for any $f\in C_c^\infty(\PSL(2,\mathbb R)/\Gamma)$ we can compute that
\begin{align*}
&\left|\frac1N\sum_{n=1}^Nf\left(u_0(-cn^{1+\gamma})p\right)-\frac1{d(q_k)}\int_0^{d(q_k)} f\left(u_0(s)q_k\right)ds\right|\\
\leq&C(N^{2+2\gamma}d_{U_-}(p,q_k)+d(q_k)^3N^{(\gamma-1)/2})\|f\|_{\infty,2}
\end{align*}
where $C$ depends only on $c$ and $\gamma$. Let $\gamma_0=0.1$ and $N=N_k:=d(q_k)^{10}$. Then for any $0<\gamma<\gamma_0$, we have $$N^{2+2\gamma}d_{U_-}(p,q_k)\leq d(q_k)^{-28},\quad d(q_k)^3N^{(\gamma-1)/2}\leq d(q_k)^{-1.5}.$$
Therefore, as $k\to\infty$, we have $$\left|\frac1{N_k}\sum_{n=1}^{N_k}f\left(u_0(-cn^{1+\gamma})p\right)-\frac1{d(q_k)}\int_0^{d(q_k)} f\left(u_0(s)q_k\right)ds\right|\to0.$$ It is known \cite[Theorem 5]{R} that $$\frac1{d(q_k)}\int_0^{d(q_k)} f\left(u_0(s)q_k\right)ds\to\int_{G/\Gamma}fd\mu_{G/\Gamma}$$ where $\mu_{G/\Gamma}$ is the invariant probability measure on $G/\Gamma$. Hence $$\frac1{N_k}\sum_{n=1}^{N_k}f\left(u_0(-cn^{1+\gamma})p\right)\to\int_{G/\Gamma}fd\mu_{G/\Gamma}\quad(\forall f\in C_c^\infty(G/\Gamma)).$$ This implies that the orbit $\{u_0(-cn^{1+\gamma})p:n\in\mathbb N\}$ is dense in $\PSL(2,\mathbb R)/\Gamma$ for any $\gamma<0.1$.
\end{proof}

\section{Proof of Theorem~\ref{th12} for Diophantine points}\label{sieve}
In the following two sections, we prove Theorem~\ref{th12}. As discussed at the beginning of section~\ref{sparse}, without loss of generality, we may assume that $$\{u(t)\}_{t\in\mathbb R}=\{u_0(-t)\}_{t\in\mathbb R}$$ is the one-parameter unipotent subgroup in Theorem~\ref{th12}. For any $x\in\mathbb R$, we denote by $[x]$ the largest integer $\leq x$. 

The following is the Jurkat-Richert theorem about linear sieve.
\begin{theorem}{\cite[Theorem 9.7]{MN}}\label{th31}
Let $A=\{a(n)\}_{n\in\mathbb N}$ be a sequence of non-negative numbers such that $$|A|=\sum_{n=1}^\infty a(n)<\infty.$$ Let $\mathcal P$ be a set of prime numbers and for $z\geq2$, let $$P(z)=\prod_{p\in\mathcal P,p<z}p.$$ Let $$S(A,\mathcal P,z)=\sum_{n=1,(n,P(z))=1}^\infty a(n).$$ For every $n\geq1$, let $g_n(d)$ be a multiplicative function such that $$0\leq g_n(p)<1,\textup{ for all }p\in\mathcal P.$$ Define $r(d)$ by $$|A_d|=\sum_{n=1,d|n}^\infty a(n)=\sum_{n=1}^\infty a(n)g_n(d)+r(d).$$ Let $\mathcal Q$ be a finite subset of $\mathcal P$, and let $Q$ be the product of the primes in $\mathcal Q$. Suppose that, for some $\epsilon$ satisfying $0<\epsilon<1/200$, the inequality $$\prod_{p\in\mathcal P\setminus\mathcal Q, u\leq p<z}(1-g_n(p))^{-1}<(1+\epsilon)\frac{\log z}{\log u}$$ holds for all $n$ and $1<u<z$. Then for any $D\geq z$ there is the upper bound $$S(A,\mathcal P, z)<(F_0(s)+\epsilon e^{14-s})X+R$$ and for any $D\geq z^2$ there is the lower bound $$S(A,\mathcal P,z)>(f_0(s)-\epsilon e^{14-s})X-R$$ for some functions $F_0(s)$ and $f_0(s)$ where $s=\frac{\log D}{\log z}$, $F_0(s)=1+O(e^{-s})$, $f_0(s)=1-O(e^{-s})$, $$X=\sum_{n=1}^\infty a(n)\prod_{p\in P(z)}(1-g_n(p))$$ and the remainder term is $$R=\sum_{d|P(z),d<DQ}|r(d)|.$$ If there is a multiplicative function $g(d)$ such that $g_n(d)=g(d)$ for all $n$, then $$X=V(z)|A|$$ where $$V(z)=\prod_{p|P(z)}(1-g(p)).$$
\end{theorem}
\begin{remark}
The explicit expressions of the functions $F_0(s)$ and $f_0(s)$ can be found in \cite[Theorem 9.4]{MN}.
\end{remark}

The following theorem will be used to verify an assumption in Theorem~\ref{th31}.
\begin{theorem}{\cite[Theorem 6.9]{MN}}\label{th32}
For any $\epsilon>0$, there exists a number $u_1(\epsilon)>0$ such that $$\prod_{u\leq p<z}\left(1-\frac1p\right)^{-1}<(1+\epsilon/3)\frac{\log z}{\log u}$$ for any $u_1(\epsilon)\leq u<z$.
\end{theorem}

We will also need the following effective equidistribution of discrete horocycle orbits. Note that in the statement, the homogeneous space is denoted by $\Gamma\backslash G$ (the space of right $\Gamma$-cosets).

\begin{theorem}{\cite[Theorem 1.2]{Z}}\label{th33}
Let $T>K\geq 1$ and $f\in C^\infty(\Gamma\backslash G)$ such that $$\int_{\Gamma\backslash G}fd\mu=0\textup{ and }\|f\|_{\infty,4}<\infty.$$ Suppose that $q\in\Gamma\backslash G$ satisfies $$r=r(q,T):=T\cdot e^{-\dist(g_{\log T}(q))}\geq1.$$ Then we have $$\left|\frac1{T/K}\sum_{j\in\mathbb Z, 0\leq Kj<T}f(qu_0(Kj))\right|\ll\frac{K^{\frac12}\ln^{\frac32}(r+2)}{r^{\beta/2}}\|f\|_{\infty,4}$$ for some $\beta\in(0,1/4)$ depending  only on the rate in the mixing property of the unipotent flow $u(t)$ and the spectral gap of the Laplacian on $\Gamma\backslash\mathcal H^2$.
\end{theorem}
\begin{remark}
Theorem 1.2 in \cite{Z} assumes that $T>K>2$. One can verify that the condition $T>K\geq1$ works equally well in the proof of \cite[Theorem 1.2]{Z}.
\end{remark}



\begin{proof}[Proof of Theorem~\ref{th12} for Diophantine points $p$ of type $\kappa=1-1/10000$]
Let $p\in G/\Gamma$ which is not $U$-periodic. Suppose that $p$ is Diophantine of type $\kappa=1-1/10000$. We prove that there exists $L\in\mathbb N$ such that the orbit $\{u_0(-x)\cdot p: x\in\Omega(L)\}$ is dense in $G/\Gamma$.

We consider the point $p^{-1}\in\Gamma\backslash G$. Pick a non-negative compactly supported smooth function $f\neq0$ on $\Gamma\backslash G$ and fix a sufficiently large number $N\in\mathbb N$. Define 
$$\begin{cases} a(n)=f(p^{-1}\cdot u_0(n)) & 0\leq n\leq N \\ a(n)=0 & n>N\end{cases}.$$ We take the set $\mathcal P$ in Theorem~\ref{th31} to be the set of all prime numbers. Then using the notation in Theorem~\ref{th31} and the result in Theorem~\ref{th33}, we have 
\begin{align*}
|A_d|=&\sum_{\substack{1\leq n\leq N\\d|n}}a(n)=\sum_{\substack{1\leq n\leq N\\d|n}}f(p^{-1}\cdot u_0(n))=\sum_{1\leq dj\leq N}f(p^{-1}\cdot u_0(jd))\\
=&\frac Nd\int_{\Gamma\backslash G}fd\mu+O\left(\left(\frac Nd\right)\frac{d^\frac12\ln^{3/2}(r+2)}{r^{\beta/2}}\|f\|_{\infty,4}\right)+O\left(\|f\|_{\infty,4}\right)\\
=&\sum_{0\leq n\leq N}\frac1d a(n)+O\left(\left(\frac Nd\right)\frac{d^\frac12\ln^{3/2}(r+2)}{r^{\beta/2}}\|f\|_{\infty,4}\right)+O\left(\left(\frac Nd\right)\frac{\ln^{3/2}(r+2)}{r^{\beta/2}}\|f\|_{\infty,4}\right)+O\left(\|f\|_{\infty,4}\right)\\
=&\sum_{1\leq n\leq N}\frac1d a(n)+O\left(\left(\left(\frac Nd\right)\frac{2d^\frac12\ln^{3/2}(r+2)}{r^{\beta/2}}+1\right)\|f\|_{\infty,4}\right).
\end{align*}
Note that here by Lemma~\ref{l31}, we have $r=r(p^{-1},N)\gg N^{1-\kappa}>1$. In Theorem~\ref{th31}, we take $g_n(d)=g(d)=1/d$ and then we have $$|r(d)|\ll\left(\left(\frac Nd\right)\frac{2d^\frac12\ln^{3/2}(r+2)}{r^{\beta/2}}+1\right)\|f\|_{\infty,4}.$$ 

Now set $z=N^\alpha$ for some small constant $\alpha>0$ which will be determined later. Let $s>100$ be a sufficiently large number so that $f_0(s)>0.1$, where $f_0(s)$ is defined as in Theorem~\ref{th31}. Fix $\epsilon\in(0,1/200)$ and take $u_1(3\epsilon)>0$ as in Theorem~\ref{th32}. Let $\mathcal Q$ be the subset of primes $<u_1(3\epsilon)$. Then by Theorem~\ref{th32}, the inequality $$\prod_{p\in\mathcal P\setminus\mathcal Q, u\leq p<z}(1-g(p))^{-1}<(1+\epsilon)\frac{\log z}{\log u}$$ holds for $1<u<z$. By Theorem~\ref{th31}, we have $$S(A,\mathcal P,z)>0.01V(z)|A|-R$$ where $D=z^s=N^{\alpha s}$. By Theorem~\ref{th32}, Theorem~\ref{th33} and Lemma~\ref{l31} we can compute that $$|A|/N\sim \int_{\Gamma\backslash G}fd\mu\textup{ and }V(z)\gg1/\log N\textup{ for sufficiently large }N.$$ Moreover, if $\alpha$ is chosen to be sufficiently small, then there exists a constant $\delta>0$ such that $$|R|\ll N^{1-\delta}.$$ Hence $S(A,\mathcal P,z)>0$ if $N$ is sufficiently large. Note that in the formula of $S(A,\mathcal P,z)$, any $n\leq N$ coprime to $P(z)$ has at most $[1/\alpha]+1$ prime factors. If we take $L=[1/\alpha]+1$ and $\Omega(L)$ the set of $L$-almost primes, then we obtain that $$\sum_{n\in\Omega(L)} f(p^{-1}\cdot u_0(n))\geq S(A,\mathcal P,z)>0 \textup{ for sufficiently large $N$}.$$ Since $f\neq0$ is any non-negatively compactly supported smooth function, we conclude that the orbit $\{p^{-1}u_0(x):x\in\Omega(L)\}$ is dense in $\Gamma\backslash G$. Consequently, the orbit $\{u_0(-x)p:x\in\Omega(L)\}$ is dense in $G/\Gamma$. This completes the proof of Theorem~\ref{th12} when $p$ is a Diophantine point of type $\kappa=1-1/10000$.
\end{proof}

\section{Proof of Theorem~\ref{th12} for non-Diophantine points}\label{nonDiophantine}
In this section, we assume that $p\in G/\Gamma$ is neither $U$-periodic nor Diophantine of type $\kappa=1-1/10000$, and prove that there exists $L\in\mathbb N$ such that the orbit $\{u_0(-x)p: x\in\Omega(L)\}$ is dense in $G/\Gamma$.

We need the following two theorems.

\begin{theorem}[{\cite{KM} and \cite[Lemma 9.7]{V}}]
Let $\sigma_j\in\Delta$ where $\Delta$ is the finite set defined in Theorem~\ref{domains}, and denote by $l_j$ the period of the $U$-periodic orbit $U\cdot\sigma_j\Gamma$. Then there exists a constant $\rho>0$ such that for any $f\in C^\infty(G/\Gamma)$ with $\|f\|_{\infty,1}<\infty$ and any $t\geq 0$, we have $$\left|\frac1{l_j}\int_0^{l_j} f(a_{-t}u_0(s)\sigma_j\Gamma)ds-\int fd_{\mu_{G/\Gamma}}\right|\ll e^{-\rho t}\|f\|_{\infty,1}.$$ Here the implicit constant depends only on $\sigma_j$, $G$ and $\Gamma$.
\end{theorem}

\begin{theorem}[{\cite[Corollary 6.2]{D1}}]
The set of $U$-periodic points in $G/\Gamma$ is equal to $\bigcup_{j=1}^nAU\sigma_j\Gamma$ where $\Delta=\{\sigma_j\}_{j=1}^n$ is the finite set defined in Theorem~\ref{domains}.
\end{theorem}

Combining these two results, we obtain the following
\begin{theorem}\label{l41}
There exists a constant $\rho>0$ such that for any $U$-periodic point $q$ with period $l_q\geq 1$ and any $f\in C^\infty(G/\Gamma)$ with $\|f\|_{\infty,1}<\infty$ we have $$\left|\frac1{l_q}\int_{0}^{l_q}f(u_0(s)q)ds-\int_{G/\Gamma}fd\mu_{G/\Gamma}\right|\ll\frac1{l_q^{\rho}}\|f\|_{\infty,1}.$$ Here the implicit constant depends only on $G/\Gamma$.
\end{theorem}

The following theorem plays an important role in the proof of Theorem~\ref{th12}.
\begin{theorem}\label{l32}
Let $f\in C^\infty(G/\Gamma)$ with $\|f\|_{\infty,1}<\infty$ and $\int_{G/\Gamma} fd\mu_{G/\Gamma}=0$, and let $q$ be a $U$-periodic point in $G/\Gamma$ with period $l_q\geq 1$. Then there exists a constant $\theta>0$ (independent of $q$ and $f$) such that for any $T\geq K>0$ with $1/{l_q^{\theta}}+{l_q}/{T^{\frac12}}<0.01K$, we have $$\left|\sum_{0<Kj<T} f(u_0(-Kj)q)\right|\ll\left(1+\sqrt{\frac {T^2}K\left(\frac{1}{l_q^{\theta}}+\frac{l_q}{\sqrt{T}}\right)}\right)\|f\|_{\infty,1}.$$ The implicit constant depends only on $G$ and $\Gamma$.
\end{theorem}

The proof of Theorem~\ref{l32} is similar to the proof of \cite[Theorem 3.1]{V} and the proof of \cite[Theorem 1.2]{Z}. We first prove the following

\begin{lemma}[Cf.~{\cite[Lemma 3.1]{V} and \cite[Lemma 5.1]{Z}}]\label{l51}
There exists a constant $\theta>0$ which only depends on $G$ and $\Gamma$ such that for any $f\in C^\infty(G/\Gamma)$ with $\|f\|_{\infty,1}<\infty$ and $\int_{G/\Gamma} fd\mu_{G/\Gamma}=0$, any character $\psi:\mathbb R\to S^1$ and any $U$-periodic point $q$ with period $l_q\geq1$, we have $$\left|\frac1T\int_0^T\psi(t)f(u_0(-t)\cdot q)\right|\ll\left(\frac1{l_q^{\theta}}+\frac{l_q}{T^{\frac12}}\right)\|f\|_{\infty,1}.$$ The implicit constant depends only on $G$ and $\Gamma$.
\end{lemma}
\begin{proof}
Let $$\mu_{T,\psi}(f)=\frac1T\int_0^T\psi(t)f(u_0(-t)\cdot q)dt$$ and for any $H>0$ define $$\sigma_H(f)(x)=\frac1H\int_0^H\psi(s)f(u_0(-s)\cdot x)ds.$$ It follows from a direct calculation that $$|\mu_{T,\psi}(f)-\mu_{T,\psi}(\sigma_H(f))|\leq\frac HT\|f\|_{\infty,0}.$$ By Cauchy-Schwartz inequality, we have 
\begin{align*}
|\mu_{T,\psi}(\sigma_H(f))|\leq&\frac1T\left(\int_0^T|\psi(t)|^2dt\right)^{\frac12}\left(\int_0^T|\sigma_H(f)(u_0(-t)\cdot q)|^2dt\right)^{\frac12}\\
\leq&\left(\frac1T\int_0^T|\sigma_H(f)(u_0(-t)\cdot q)|^2dt\right)^{\frac12}\\
\leq&\left(\frac1{H^2}\int_0^H\int_0^H\left|\frac1T\int_0^T\overline{f^y}f^z(u_0(-t)q)dt\right|dydz\right)^{\frac12}
\end{align*}
where $f^y$ and $f^z$ denote the translates of $f$ by $u_0(-y)$ and $u_0(-z)$ respectively. By Theorem~\ref{l41}, there exists $\rho>0$ such that $$\left|\frac1T\int_0^T\overline{f^y}f^z(u_0(-t)q)dt-\int_{G/\Gamma}\overline{f^y}f^zd\mu_{G/\Gamma}\right|\ll\left(\frac1{l_q^\rho}+\frac{l_q}T\right)\|\overline{f^y}f^z\|_{\infty,1}.$$ By the mixing property of the unipotent flow $u_0(-t)$, there exists $\lambda>0$ such that $$\left|\int_{G/\Gamma}\overline{f^h}fd\mu_{G/\Gamma}\right|\ll\frac1{(1+|h|)^\lambda}\|f\|_{\infty,1}^2.$$ Then we can compute that
\begin{align*}
|\mu_{T,\psi}(\sigma_H(f))|\ll&\left(\frac1{H^2}\int_0^H\int_0^H\left|\int_{G/\Gamma}\overline{f^{y-z}}fd\mu_{G/\Gamma}\right|dydz+\frac1{H^2}\int_0^H\int_0^H\left(\frac1{l_q^\rho}+\frac{l_q}T\right)\|\overline{f^y}f^z\|_{\infty,1}dydz\right)^{\frac12}\\
\ll&\left(\frac1{H^2}\int_0^H\int_0^H\frac1{(1+|y-z|)^\lambda}dydz+\frac1{H^2}\int_0^H\int_0^H\left(\frac1{l_q^\rho}+\frac{l_q}T\right)yzdydz\right)^{\frac12}\|f\|_{\infty,1}\\
\ll&\left(H^{-\lambda}+H^2\left(\frac1{l_q^\rho}+\frac{l_q}T\right)\right)^{\frac12}\|f\|_{\infty,1}\leq\left(H^{-\lambda/2}+H\left(\frac1{l_q^{\rho/2}}+\frac{l_q^{1/2}}{T^{1/2}}\right)\right)\|f\|_{\infty,1}.
\end{align*}
We complete the proof of the lemma by setting $H=l_q^\epsilon$ for some sufficiently small $\epsilon>0$. 
\end{proof}

\begin{proof}[Proof of Theorem~\ref{l32}]
Let $0<\delta<0.5K$. Define $$g_\delta(x)=\max\left\{\frac1{\delta^2}(\delta-|x|),0\right\}\textup{ and }g(x)=\sum_{j\in\mathbb Z}g_{\delta}(x+Kj).$$ Write the Fourier series of $g(x)$ as $$g(x)=\sum_{k\in\mathbb Z}a_ke^{2\pi ikx/K}.$$ Then one can compute that $$\sum_{k\in\mathbb Z}|a_k|=|g(0)|=\frac1\delta.$$ It follows from a direct calculation that $$\left|\int_0^Tg(t)f(u_0(-t)\cdot q)dt-\sum_{0<Kj<T}f(u_0(-Kj)\cdot q)\right|\leq\left(2+\frac TK\cdot\delta\right)\|f\|_{\infty,1}.$$ By Lemma~\ref{l51}, there exists $\theta>0$ such that
\begin{align*}
\left|\int_0^Tg(t)f(u_0(-t)\cdot q)dt\right|\leq&\sum_k|a_k|\left|\int_0^Te^{2\pi ikt/K}f(u_0(-t)\cdot q)dt\right|\\
\ll& \frac T\delta\left(\frac1{l_q^{\theta}}+\frac{l_q}{T^{\frac12}}\right)\|f\|_{\infty,1}.
\end{align*}
Therefore, we have $$\left|\sum_{0<Kj<T}f(u_0(-Kj)\cdot q)\right|\ll\left(\frac T\delta\left(\frac1{l_q^{\theta}}+\frac{l_q}{T^{\frac12}}\right)+2+\frac TK\cdot\delta\right)\|f\|_{\infty,1}.$$
Set $\delta=\sqrt{K\left(\frac1{l_q^{\theta}}+\frac{l_q}{T^{\frac12}}\right)}$ and we conclude that $$\left|\sum_{0<Kj<T}f(u_0(-Kj)\cdot q)\right|\ll\left(\sqrt{\frac {T^2}K\left(\frac1{l_q^{\theta}}+\frac{l_q}{T^{\frac12}}\right)}+1\right)\|f\|_{\infty,1}.$$ This completes the proof of Theorem~\ref{l32}.
\end{proof}

By Bruhat decomposition, we know that $$G=UAU_-\cup\omega AU_-\textup{ and }G/\Gamma=UAU_-/\Gamma\cup\omega AU_-/\Gamma.$$ Since any point in $\omega AU_-/\Gamma$ is a $U$-periodic point, we have $p\in UAU_-\Gamma$. Now write $p=uan\Gamma$ for $u\in U$, $a\in A$ and $n\in U_-$. Since $a$ normalizes $U$, we have $$u_0(t)\cdot uan\Gamma=(ua)\cdot u_0(ct)\cdot n\Gamma$$ for some constant $c>0$. Hence to prove Theorem~\ref{th12} for the point $p$ which is not Diophantine of type $\kappa=1-1/10000$, it suffices to prove the following

\begin{proposition}\label{p31}
There exists $L\in\mathbb N$ such that for any $p\in U_-(e\Gamma)$ which is neither $U$-periodic nor Diophantine of type $\kappa=1-1/10000$, and for any $c>0$, the orbit $\{u_0(-cx)p: x\in\Omega(L)\}$ is dense in $G/\Gamma$.
\end{proposition}
\begin{proof}[Proof of Proposition~\ref{p31}]
Let $p\in U_-(e\Gamma)$ which is neither $U$-periodic nor Diophantine of type $\kappa=1-1/10000$, and fix $c>0$. By Proposition~\ref{p22}, we can find a sequence of distinct $U$-periodic points $q_k\in U_-(e\Gamma)$ with period $d(q_k)\to\infty$ such that $$d_{U_-}(p, q_k)\leq Cd(q_k)^{-\frac1{1-\kappa}}$$ for some constant $C>0$. Now choose a rational point $q_k$ and pick a non-negative compactly supported smooth function $f\neq 0$ on $G/\Gamma$. Define $$\begin{cases} a(n):=f(u_0(-cn)\cdot q_k) & n\leq N\\ a(n):=0 & n>N\end{cases}$$ and take the set $\mathcal P$ in Theorem~\ref{th31} to be the set of all prime numbers. Using the notation in Theorem~\ref{th31} and the result in Theorem~\ref{l32}, we can compute that
\begin{align*}
|A_d|=&\sum_{1\leq n\leq N,d|n}a(n)=\sum_{0<n\leq N,d|n}f(u_0(-cn)\cdot q_k)\\
=&\sum_{0<cdj\leq cN}f(u_0(-cdj)\cdot q_k)\\
=&\frac Nd\int_{G/\Gamma}fd\mu_{G/\Gamma}+O\left(\left(1+\sqrt{\frac {N^2}d\left(\frac{1}{d(q_k)^{\theta}}+\frac{d(q_k)}{\sqrt{N}}\right)}\right)\|f\|_{\infty,1}\right)\\
=&\frac1d\sum_{1\leq n\leq N}a(n)+O\left(\left(1+\sqrt{\frac {N^2}d\left(\frac{1}{d(q_k)^{\theta}}+\frac{d(q_k)}{\sqrt{N}}\right)}\right)\|f\|_{\infty,1}\right)
\end{align*}
for some $\theta>0$. (Here we remark that we will choose $N$ later so that $N$ and $d(q_k)$ meet the requirement in Theorem~\ref{l32}.) In Theorem~\ref{th31}, set $g_n(d)=g(d)=1/d$ and we have 
\begin{align*}
|r(d)|\ll&\left(1+\sqrt{\frac {N^2}d\left(\frac{1}{d(q_k)^{\theta}}+\frac{d(q_k)}{\sqrt{N}}\right)}\right)\|f\|_{\infty,1}\\
\leq&\left(1+\frac1{\sqrt d}\left(\frac N{d(q_k)^{\theta/2}}+N^{3/4}d(q_k)^{1/2}\right)\right)\|f\|_{\infty,1}.
\end{align*}
Let $z=N^\alpha$ for some small constant $\alpha>0$ which will be determined later. Let $s>100$ be a sufficiently large number so that $f_0(s)>0.1$ where $f_0(s)$ is the function defined in Theorem~\ref{th31}. Fix $\epsilon\in(0,1/200)$ and take $u_1(\epsilon)>0$ as in Theorem~\ref{th32}. Let $\mathcal Q$ be the subset of primes$<u_1(3\epsilon)$. Then by Theorem~\ref{th32}, the inequality $$\prod_{p\in\mathcal P\setminus\mathcal Q, u\leq p<z}(1-g(p))^{-1}<(1+\epsilon)\frac{\log z}{\log u}$$ holds for $1<u<z$. By Theorem~\ref{th31}, we have $$S(A,\mathcal P,z)>0.01V(z)|A|-R$$ where $D=z^s=N^{\alpha s}$. Here by Theorem~\ref{l32} and Theorem~\ref{th32}, if $N$ and $d(q_k)$ are sufficiently large, then $$|A|/N\sim\int_{G/\Gamma}fd\mu_{G/\Gamma}>0\textup{ and }V(z)\gg1/\log N.$$ Moreover, we can compute that $$R\ll\left(N^{\alpha s}+N^{\alpha s}\left(\frac N{d(q_k)^{\theta/2}}+N^{3/4}d(q_k)^{1/2}\right)\right)\|f\|_{\infty,1}.$$

Now write $p=\left(\begin{array}{cc} 1 & 0\\ x & 1\end{array}\right)\Gamma$ and $q_k=\left(\begin{array}{cc} 1 & 0\\ y_k & 1\end{array}\right)\Gamma$. For any $s\in\mathbb R$ and $s\geq1$, we calculate
\begin{align*}
u_0(s)q_k=\left(\begin{array}{cc} 1 & s\\ 0 & 1\end{array}\right)\left(\begin{array}{cc} 1 & 0\\ y_k-x & 1 \end{array}\right)\left(\begin{array}{cc} 1 & -s\\ 0 & 1\end{array}\right)u_0(s)p.
\end{align*}
Using the same calculation as in Lemma~\ref{l23}, we can compute that the distance between $u_0(s)p$ and $u_0(s)q_k$ is bounded by $$|s^2(y_k-x)|\sim s^2d_{U_-}(p,q_k).$$ Therefore, we have
\begin{align*}
|f(u_0(s)p)-f(u_0(s)q_k)|\ll\|f\|_{\infty,1}s^2d_{U_-}(p,q_k)\ll\|f\|_{\infty,1}s^2d(q_k)^{-\frac1{1-\kappa}}
\end{align*}
and
\begin{align*}
&\left|\sum_{1\leq n\leq N, (P(z),n)=1}f(u_0(-cn)p)-S(A,\mathcal P,z)\right|\\
\ll&\sum_{1\leq n\leq N, (P(z),n)=1}n^2d(q_k)^{-\frac1{1-\kappa}}\|f\|_{\infty,1}\\
\ll&N^3d(q_k)^{-\frac1{1-\kappa}}\|f\|_{\infty,1}.
\end{align*}
This implies that 
\begin{align}\label{equation}
&\sum_{1\leq n\leq N, (P(z),n)=1}f(u_0(-cn)p)\nonumber\\ 
=&S(A,\mathcal P,z)-O\left(N^3d(q_k)^{-\frac1{1-\kappa}}\|f\|_{\infty,1}\right)\nonumber\\
\geq&0.01V(z)|A|-O\left(\left(N^{\alpha s}+N^{\alpha s}\left(\frac N{d(q_k)^{\theta/2}}+N^{3/4}d(q_k)^{1/2}\right)\right)+N^3d(q_k)^{-\frac1{1-\kappa}}\right)\|f\|_{\infty,1}.
\end{align}
Let $N:=N_k=d(q_k)^{100}$. Then if $\alpha$ is chosen to be sufficiently small, then we have $$N^{\alpha s}\ll N^{1-\epsilon}, N^{\alpha s}\frac N{d(q_k)^{\theta/2}}\ll N^{1-\epsilon}, N^{\alpha s}N^{3/4}d(q_k)^{1/2}\ll N^{1-\epsilon}, N^3d(q_k)^{-\frac1{1-\kappa}}\ll N^{1-\epsilon}$$ for some $\epsilon>0$. Moreover, $N=N_k$ and $d(q_k)$ meet the requirement in Theorem~\ref{l32} if $k$ is sufficiently large. Therefore, the main term in inequality~(\ref{equation}) is $0.01V(z)|A|$ and $$\sum_{1\leq n\leq N_k, (P(z),n)=1}f(u_0(-cn)p)>0$$ for sufficiently large $k$ (and for any non-negative compactly supported smooth function $f\neq0$). Note that any $n$ coprime to $P(z)$ $(z=N_k^\alpha)$ has at most $[1/\alpha]+1$ prime factors. If we take $L=[1/\alpha]+1$ and $\Omega(L)$ the set of $L$-almost primes, then we obtain that $$\sum_{n\in\Omega(L)}f(u_0(-cn)p)>0$$ for any non-negative compactly supported smooth function $f\neq 0$. This implies that $\{u_0(-cx)p:x\in\Omega(L)\}$ is dense in $\PSL(2,\mathbb R)/\Gamma$. This completes the proof of Proposition~\ref{p31}.
\end{proof}

\begin{proof}[Proof of Theorem~\ref{th12}]
Note that if $L_1\geq L_2$, then $\Omega(L_1)\supset\Omega(L_2)$. Now combining results in \S\ref{sieve} and \S\ref{nonDiophantine}, we can conclude that there exists $L\in\mathbb N$ such that for any $p\in G/\Gamma$ which is not $u(t)$-periodic, the subset $\{u(x)p: x\in\Omega(L)\}$ is dense in $G/\Gamma$.
\end{proof}


\begin{bibdiv}
\begin{biblist}

\bib{B1}{article}{
      author={Bourgain, J.},
       title={On the pointwise ergodic theorem on {$L^p$} for arithmetic sets},
        date={1988},
        ISSN={0021-2172},
     journal={Israel J. Math.},
      volume={61},
      number={1},
       pages={73\ndash 84},
         url={https://doi.org/10.1007/BF02776302},
      review={\MR{937582}},
}

\bib{D}{article}{
      author={Dani, S.~G.},
       title={Invariant measures and minimal sets of horospherical flows},
        date={1981},
        ISSN={0020-9910},
     journal={Invent. Math.},
      volume={64},
      number={2},
       pages={357\ndash 385},
         url={https://doi.org/10.1007/BF01389173},
      review={\MR{629475}},
}

\bib{D1}{article}{
 author = {S.~G. Dani.},
 title = {Divergent trajectories of flows on homogeneous spaces and {D}iophantine approximation.},
 date = {1985},
 journal = {J. Reine Angew. Math.},
 volume = {359},
 pages = {55--89},
}


\bib{G}{article}{
      author={Gorodnik, Alexander},
       title={Open problems in dynamics and related fields},
        date={2007},
        ISSN={1930-5311},
     journal={J. Mod. Dyn.},
      volume={1},
      number={1},
       pages={1\ndash 35},
         url={https://doi.org/10.3934/jmd.2007.1.1},
      review={\MR{2261070}},
}

\bib{GR}{article}{
    AUTHOR = {Garland, H.},
    AUTHOR = {Raghunathan, M. S.},
     TITLE = {Fundamental domains for lattices in ({R}-)rank {$1$}
              semisimple {L}ie groups},
   JOURNAL = {Ann. of Math. (2)},
  FJOURNAL = {Annals of Mathematics. Second Series},
    VOLUME = {92},
      YEAR = {1970},
     PAGES = {279--326},
      ISSN = {0003-486X},
   MRCLASS = {22.50 (10.00)},
  review = {\MR{267041}},
MRREVIEWER = {J. A. Wolf},
       URL = {https://doi.org/10.2307/1970838},
}

\bib{GT}{article}{
    AUTHOR = {Green, Ben},
    AUTHOR = {Tao, Terence},
     TITLE = {The quantitative behaviour of polynomial orbits on
              nilmanifolds},
   JOURNAL = {Ann. of Math. (2)},
  FJOURNAL = {Annals of Mathematics. Second Series},
    VOLUME = {175},
      YEAR = {2012},
    NUMBER = {2},
     PAGES = {465--540},
      ISSN = {0003-486X},
   MRCLASS = {37A15},
  review = {\MR{2877065}},
MRREVIEWER = {Tamar Ziegler},
       URL = {https://doi.org/10.4007/annals.2012.175.2.2},
}

\bib{IK}{book}{
      author={Iwaniec, Henryk},
      author={Kowalski, Emmanuel},
       title={Analytic number theory},
      series={American Mathematical Society Colloquium Publications},
   publisher={American Mathematical Society, Providence, RI},
        date={2004},
      volume={53},
        ISBN={0-8218-3633-1},
         url={https://doi.org/10.1090/coll/053},
      review={\MR{2061214}},
}

\bib{K1}{article}{
      author={Katz, Asaf},
       title={On mixing and sparse ergodic theorems},
        date={2021},
        ISSN={1930-5311},
     journal={J. Mod. Dyn.},
      volume={17},
       pages={1\ndash 32},
         url={https://doi.org/10.3934/jmd.2021001},
      review={\MR{4251936}},
}

\bib{K2}{article}{
      author={Khalil, Osama},
       title={Pointwise equidistribution and translates of measures on
  homogeneous spaces},
        date={2020},
        ISSN={0143-3857},
     journal={Ergodic Theory Dynam. Systems},
      volume={40},
      number={2},
       pages={453\ndash 477},
         url={https://doi.org/10.1017/etds.2018.44},
      review={\MR{4048301}},
}

\bib{KM}{article}{
    AUTHOR = {Kleinbock, D. Y.},
    AUTHOR = {Margulis, G. A.},
     TITLE = {Logarithm laws for flows on homogeneous spaces},
   JOURNAL = {Invent. Math.},
  FJOURNAL = {Inventiones Mathematicae},
    VOLUME = {138},
      YEAR = {1999},
    NUMBER = {3},
     PAGES = {451--494},
      ISSN = {0020-9910},
   MRCLASS = {37C85 (11J99 22E40 37A30 37D30 37D40)},
  review = {\MR{1719827}},
MRREVIEWER = {S. G. Dani},
       URL = {https://doi.org/10.1007/s002220050350},
}


\bib{L}{article}{
    AUTHOR = {Leibman, A.},
     TITLE = {Pointwise convergence of ergodic averages for polynomial
              sequences of translations on a nilmanifold},
   JOURNAL = {Ergodic Theory Dynam. Systems},
  FJOURNAL = {Ergodic Theory and Dynamical Systems},
    VOLUME = {25},
      YEAR = {2005},
    NUMBER = {1},
     PAGES = {201--213},
      ISSN = {0143-3857},
   MRCLASS = {37A17 (22F30 28D15)},
 review = {\MR{2122919}},
MRREVIEWER = {Alexander Gorodnik},
       URL = {https://doi.org/10.1017/S0143385704000215},
}

\bib{M1}{incollection}{
      author={Margulis, G.~A.},
       title={Discrete subgroups and ergodic theory},
        date={1989},
   booktitle={Number theory, trace formulas and discrete groups ({O}slo,
  1987)},
   publisher={Academic Press, Boston, MA},
       pages={377\ndash 398},
      review={\MR{993328}},
}

\bib{MT}{article}{
      author={Margulis, G.~A.},
      author={Tomanov, G.~M.},
       title={Invariant measures for actions of unipotent groups over local
  fields on homogeneous spaces},
        date={1994},
        ISSN={0020-9910},
     journal={Invent. Math.},
      volume={116},
      number={1-3},
       pages={347\ndash 392},
         url={https://doi.org/10.1007/BF01231565},
      review={\MR{1253197}},
}

\bib{M2}{article}{
      author={McAdam, Taylor},
       title={Almost-prime times in horospherical flows on the space of
  lattices},
        date={2019},
        ISSN={1930-5311},
     journal={J. Mod. Dyn.},
      volume={15},
       pages={277\ndash 327},
         url={https://doi.org/10.3934/jmd.2019022},
      review={\MR{4042163}},
}

\bib{MN}{book}{
      author={Nathanson, Melvyn~B.},
       title={Additive number theory},
      series={Graduate Texts in Mathematics},
   publisher={Springer-Verlag, New York},
        date={1996},
      volume={164},
        ISBN={0-387-94656-X},
         url={https://doi.org/10.1007/978-1-4757-3845-2},
        note={The classical bases},
      review={\MR{1395371}},
}

\bib{R1}{article}{
      author={Ratner, Marina},
       title={On {R}aghunathan's measure conjecture},
        date={1991},
        ISSN={0003-486X},
     journal={Ann. of Math. (2)},
      volume={134},
      number={3},
       pages={545\ndash 607},
         url={https://doi.org/10.2307/2944357},
      review={\MR{1135878}},
}

\bib{R}{article}{
    author = {Ratner, Marina},
    title = {Raghunathan's conjectures for {${\rm SL}(2,\bold R)$}},
   date={1992},
   ISSN = {0021-2172},
   journal = {Israel J. Math.},
   volume = {80},
   number = {1-2},
    pages = {1\ndash 31},
    url = {https://doi.org/10.1007/BF02808152},
   review = {\MR{1248925}},
}

\bib{R2}{article}{
      author={Ratner, Marina},
       title={Raghunathan's topological conjecture and distributions of
  unipotent flows},
        date={1991},
        ISSN={0012-7094},
     journal={Duke Math. J.},
      volume={63},
      number={1},
       pages={235\ndash 280},
         url={https://doi.org/10.1215/S0012-7094-91-06311-8},
      review={\MR{1106945}},
}

\bib{SU}{article}{
      author={Sarnak, Peter},
      author={Ubis, Adri\'{a}n},
       title={The horocycle flow at prime times},
        date={2015},
        ISSN={0021-7824},
     journal={J. Math. Pures Appl. (9)},
      volume={103},
      number={2},
       pages={575\ndash 618},
         url={https://doi.org/10.1016/j.matpur.2014.07.004},
      review={\MR{3298371}},
}

\bib{Sh}{article}{
      author={Shah, Nimish~A.},
       title={Limit distributions of polynomial trajectories on homogeneous
  spaces},
        date={1994},
        ISSN={0012-7094},
     journal={Duke Math. J.},
      volume={75},
      number={3},
       pages={711\ndash 732},
         url={https://doi.org/10.1215/S0012-7094-94-07521-2},
      review={\MR{1291701}},
}

\bib{TV}{article}{
      author={Tanis, James},
      author={Vishe, Pankaj},
       title={Uniform bounds for period integrals and sparse equidistribution},
        date={2015},
        ISSN={1073-7928},
     journal={Int. Math. Res. Not. IMRN},
      number={24},
       pages={13728\ndash 13756},
         url={https://doi.org/10.1093/imrn/rnv115},
      review={\MR{3436162}},
}

\bib{V}{article}{
      author={Venkatesh, Akshay},
       title={Sparse equidistribution problems, period bounds and
  subconvexity},
        date={2010},
        ISSN={0003-486X},
     journal={Ann. of Math. (2)},
      volume={172},
      number={2},
       pages={989\ndash 1094},
         url={https://doi.org/10.4007/annals.2010.172.989},
      review={\MR{2680486}},
}

\bib{Z}{article}{
      author={Zheng, Cheng},
       title={Sparse equidistribution of unipotent orbits in finite-volume
  quotients of {$\text{PSL}(2,\Bbb{R})$}},
        date={2016},
        ISSN={1930-5311},
     journal={J. Mod. Dyn.},
      volume={10},
       pages={1\ndash 21},
         url={https://doi.org/10.3934/jmd.2016.10.1},
      review={\MR{3471070}},
}

\bib{CZ}{article}{
      author={Zheng, Cheng},
       title={A shrinking target problem with target at infinity in rank one
  homogeneous spaces},
        date={2019},
        ISSN={0026-9255},
     journal={Monatsh. Math.},
      volume={189},
      number={3},
       pages={549\ndash 592},
         url={https://doi.org/10.1007/s00605-019-01309-2},
      review={\MR{3976679}},
}

\end{biblist}
\end{bibdiv}
\end{document}